\documentclass[11pt]{amsart}
\usepackage{geometry}                
\usepackage{graphicx}
\usepackage{amssymb}
\usepackage{epstopdf}
\usepackage{amsmath}
\usepackage{color}
\usepackage{hyperref}
 \usepackage{tikz}
 
\usepackage[all]{xy}
\xyoption{matrix}
\xyoption{arrow}

\usetikzlibrary{arrows,decorations.pathmorphing,backgrounds,positioning,fit,petri}

 \tikzset{help lines/.style={step=#1cm,very thin, color=gray},
help lines/.default=.5} 
\tikzset{thick grid/.style={step=#1cm,thick, color=gray},
thick grid/.default=1} 

\newtheorem{thm}{Theorem}[subsection]
\newtheorem{lem}[thm]{Lemma}
\newtheorem{cor}[thm]{Corollary}
\newtheorem{prop}[thm]{Proposition}

\theoremstyle{definition}
\newtheorem{defn}[thm]{Definition}

\theoremstyle{remark}
\newtheorem{rem}[thm]{Remark}
 
\numberwithin{equation}{section}

\newcommand{\then}{\Rightarrow}

\DeclareMathOperator{\Hom}{Hom}%
\newcommand{\field}[1]{\mathbb{#1}}
\newcommand{\ZZ}{\ensuremath{{\field{Z}}}}
\newcommand{\CC}{\ensuremath{{\field{C}}}}

\newcommand{\commentout}[1]{}

\newcommand{\cC}{\ensuremath{{\mathcal{C}}}}

\title{Weighted quivers}

\author{Kiyoshi Igusa}
\address{Department of Mathematics, Brandeis University, Waltham, MA 02454}\email{igusa@brandeis.edu}

\author{Moses Kim}
\address{Department of Mathematics, Brandeis University, Waltham, MA 02454}\email{moseskim@brandeis.edu}

\begin{document}
\maketitle

\begin{abstract}
A ``weight'' on a quiver $Q$ with values in a group $G$ is a function which assigns an element of $G$ for each arrow in $Q$. This paper shows that the essential steps in the mutation of quivers with potential [DWZ] goes through with weights provided that the weights on each cycle in the potential have trivial product. This gives another proof of the sign coherence of $c$-vectors. We also classify all weights on tame quivers.
\end{abstract}

\section{Introduction}

Let $Q$ be a finite quiver without loops and let $G$ be any group. A \emph{weight system} for $Q$ is a function assigning an element of $G$ to every arrow in $Q$. A weighted quiver without oriented 2-cycles can be mutated at any vertex $k$ as follows.
\begin{enumerate}
\item For every pair of arrows $a:i\to k$, $b:k\to j$ add a new arrow $[ab]:i\to j$. Give the new arrow the product weight:
\[
	wt[ab]:=wt(a)wt(b)
\]
\item Reverse the orientation of all arrows to and from $k$ and invert the weight of each of these arrows. (This gives $a^\ast:k\to i$ and $b^\ast:j\to k$ with weights $wt(a^\ast)=wt(a)^{-1}$, $wt(b^\ast)=wt(b)^{-1}$.)
\item Remove oriented 2-cycles of trivial weight ($a:i\to j$, $b:j\to i$ with $wt(a)wt(b)=1$).
\end{enumerate}
If all oriented 2-cycles are eliminated, the new weighted quiver $(Q',wt')$ can be mutated again.

The idea is that $G$ is the fundamental group of a pointed space $X$: $G=\pi_1(X,x_0)$ and the weights give a mapping from the quiver $Q$, considered as a 1-dimensional CW-complex to $X$ sending all vertices to the basepoint $x_0\in X$. Each oriented edge gives a loop in $X$ and therefore an element of $G=\pi_1(X)$. When the orientation of the edge is reversed, the element of $\pi_1(X)$ is inverted. In Step 3, 2-cycles can be eliminated iff they are null homotopic.

We say that $(Q,wt)$ is \emph{nondegenerate} if no sequence of mutations produces an oriented 2-cycle with nontrivial weight, i.e., if Step 3 always eliminates all 2-cycles.

One of the main applications of weights on quivers is a new elementary proof of the sign-coherence of $c$-vectors. The proof is given by the following outline.
\begin{enumerate}
\item Start with any finite quiver $Q$.
\item By \cite{DWZ} there is a nondegenerate potential $S$.
\item Choose any system of weight compatible with $S$, i.e., so that every term in $S$ has weight 1. For example, take trivial weights on all edges of all oriented cycles.
\item Lemma: Any sequence of mutations can be performed respecting the weights. In particular, the 2-cycles which are eliminated in the process of mutation of quivers with potential always have trivial weight.
\item This proves that the weighted quiver $(Q,wt)$ is nondegenerate.
\end{enumerate}

There is no formula for the generic potential $S$. However, we know that every term in $S$ consists of oriented cycles. This proves the following.

\begin{thm}\label{thm: nondegenerate potential gives nondegenerate weight for Q}
Let $Q$ be a quiver without loops or 2-cycles and let $wt$ be a system of weights on $Q$ gives every oriented cycle the trivial weight 1. Then $(Q,wt)$ is nondegenerate.
\end{thm}

This implies sign-coherence of $c$-vectors which we state as follows.

\begin{cor}\label{cor: sign coherence of c-vectors}
Let $Q$ be a quiver with frozen vertices so that each frozen vertex is a source. Then, no sequence of mutations at unfrozen vertices can produce an arrow between two frozen vertices.
\end{cor}

\begin{proof}
Suppose there is a sequence of mutations on $Q$ producing an arrow between frozen vertices $i\to j$. Let $Q'$ be the quiver obtained from the original quiver $Q$ by adding one arrow $a:j\to i$. Give weights to $Q'$ as follows. All arrows have weight 1 except for the new arrow $a$ which has nontrivial weight. Since $a$ is not part of any oriented cycle, $(Q',wt)$ is nondegenerate. However, the given sequence of mutations produces an arrow $b:i\to j$ of weight 1. Together with $a:i\to j$ we get an oriented 2-cycle of nontrivial weight which contradicts the assumption that $(Q',wt)$ is nondegenerate.
\end{proof}

\begin{rem}\label{rem: equivalent wording of sign coherence}
The standard wording of the sign coherence of $c$-vectors is a special case of Corollary \ref{cor: sign coherence of c-vectors}. If $1',2',\cdots,n'$ are the frozen vertices of $Q$ and $k$ is a not frozen vertex of $Q$ then the $c$-vector of $k$ is defined to be the integer vector $c_k\in\ZZ^n$ with coordinates $c_{kj}$ equal to the number of arrow from $k$ to the $j$th frozen vertex $j'$ minus the number of arrows from $j'$ to $k$. The vector $c_k$ is \emph{sign coherent} if its entries are either all nonnegative or all nonpositive. One version of the sign coherence conjecture says that, if the frozen vertices are all sources, then any sequence of mutations at not frozen vertices will keep the $c$-vectors sign coherent. If this is violated, then there will be a $c$-vector $c_k$ with two entries, say $c_{ki}$ and $c_{kj}$ of opposite sign. But, if this happens, mutation at vertex $k$ will produce an arrow between frozen vertices $i'$ and $j'$ violating Corollary \ref{cor: sign coherence of c-vectors}. Conversely, any violation of Corollary \ref{cor: sign coherence of c-vectors} must be preceded by such a violation of sign coherence of $c$-vectors. So, the statements are equivalent.
\end{rem}

As another application, we give a classification of all nondegenerate weights on all tame quivers, i.e., those mutation equivalent to tame acyclic quivers.

\section{Quivers with potentials and weights}

In this section we go over the definitions and proofs in \cite{DWZ} to check that they work without any problems when weights are added. We use the same notation as in \cite{DWZ} with the exception that we compose arrow from left to right and we work with a fixed field $K=\CC$.

\subsection{Quivers with weights}

Let $Q=(Q_0,Q_1,wt)$ be a quiver with finite \emph{vertex set}
\[
	Q_0=\{1,2,\cdots,n\}
\]
finite \emph{arrow set} $Q_1$ and \emph{weight} function
\[
	wt:Q_1\to G
\]
where $G$ is a fixed group.

We work over the field of complex numbers $\CC$. The \emph{vertex span} of $Q$ is $R=\CC^{Q_0}$, the vector space spanned by the vertices of $Q$. This is a semi-simple algebra over $\CC$ with idempotents $e_i$, $i=1,\cdots,n$.

The \emph{arrow span} of $Q$ is $A=\CC^{Q_1}$. This decomposes as
\begin{gather}
	A=\bigoplus_{g\in G} A_g
\end{gather}
where $A_g$ is the span of all arrows of $Q$ of weight $g$. $A$ is an $R$-bimodule. We use the topologist convention of composing arrows left to right. So, $e_iAe_j$ is the span of all arrows $i\to j$. For each $g\in G$, $A_g$ is a sub-bimodule of $A$. By a \emph{homogeneous basis} for $A$ we mean a subset $B$ of $A$ so that $B=\coprod B_g^{ij}$ where $B_g^{ij}=B\cap e_iA_ge_j$ is a basis for $e_iA_ge_j$ for every $i,j\in Q_0$, $g\in G$.

The \emph{dual} of $Q$ is given by reversing the direction of all arrows and inverting the weights.

For every $d\ge0$, $A^d=A\otimes_R A\otimes_R\cdots\otimes_RA$ is the span of all paths of length $d$ in $Q$. Thus, $A^0=R$ and $A^d_g$ is the sub-bimodule of all paths with weight $g$ where the \emph{weight of a path} is defined to be the product of the weights of the arrows:
\[
	wt(a_1\otimes a_2\otimes\cdots\otimes a_d)= wt(a_1)wt(a_2)\cdots wt(a_d)
\]
and $wt(e_i)=1$ for all $i$. We often suppress the tensor product symbol and write $a_1a_2\cdots a_d$.

The \emph{path algebra} of $Q$ is the direct sum of all $A^d$:
\[
	R\left<A\right>=\bigoplus_{d\ge0} A^d
\]
This is a $G$-graded algebra where $R\left<A\right>_g=\bigoplus_{d\ge0} A^d_g$ for all $g\in G$. The \emph{completed path algebra} of $Q$ is denoted
\[
	R\left<\left<A\right>\right>=\prod_{d\ge0} A^d
\]
This is also $G$-graded: $R\left<\left<A\right>\right>_g=\prod_{d\ge0} A^d_g$. The \emph{radical} of $R\left<\left<A\right>\right>$ is $\frak m=\prod_{d\ge1} A^d$. Then $R\left<\left<A\right>\right>/\frak m=R$. The radical is $G$-graded: $\frak m=\bigoplus \frak m_g$ where $\frak m_g=\prod_{d\ge1}A^d_g$.

\subsection{Automorphisms of quivers with weights}

When $Q,Q'$ are quivers with the same vertex set and arrow spans $A,A'$, we consider $G$-graded $\CC$-algebra homomorphisms $\varphi:R\left<\left<A\right>\right>\to R\left<\left<A'\right>\right>$ which are the identity on $R$. Thus $\varphi(A_g)\subseteq R\left<\left<A'\right>\right>_g$ for all $g\in G$. In fact, $\varphi$ is uniquely determined by $R$-bimodule morphisms
\[
	\varphi_g:A_g\to \frak m_g'\subset R\left<\left<A'\right>\right>_g
\]
which are arbitrary where $\frak m'$ is the radical of $R\left<\left<A'\right>\right>$ and $\frak m_g'=\frak m'\cap R\left<\left<A'\right>\right>_g$. Note that an arbitrary $R$-bilinear map $A_g\to \frak m_g'$ is given by a collection of linear maps:
\[
	\varphi_g^{ij}:A_g^{ij}\to (\frak m'_g)^{ij}
\]
for all $i,j\in Q_0,g\in G$.

Following \cite{DWZ} we call an automorphism $\varphi$ of $R\left<\left<A\right>\right>$ a \emph{change of arrows} if $\varphi(A_g^{ij})= A_g^{ij}$, i.e., $\varphi$ is given by a linear change of coordinates of each $A_g^{ij}$. If $\varphi$ is the identity modulo $\frak m^2$, we call it \emph{unitriangular}.

\subsection{Potentials}

For $d\ge 1$, the \emph{cyclic part} of $A^d$, denoted $A^d_{cyc}$, is defined to be the sub-$R$-bimodule of $A^d_1$ (the weight 1 part of $A^d$) spanned by oriented cycles of length $d$.
\[
	R\left<\left<A\right>\right>_{cyc}=\prod_{d\ge1} A^d_{cyc}
\]
The concept of \emph{cyclic equivalence} is defined as before. Thus $a b$ and $ba$ are cyclically equivalent if $a,b,ab,ba$ are paths. Elements of $R\left<\left<A\right>\right>_{cyc}$ (including 0) are called \emph{potentials}. 

It is clear that every $G$-graded algebra homomorphism $\varphi:R\left<\left<A\right>\right>\to R\left<\left<A'\right>\right>$ sends potentials to potentials.

\subsection{Weighted quivers with potential}

A \emph{weighted quiver with potential} (wQP), denoted $(A,S)$ is a weighted quiver $Q$ without loops, and a potential $S\in R\left<\left<A\right>\right>_{cyc}$ no two terms of which are cyclically equivalent up to a scalar. I.e., $S$ is a (possibly infinite) linear combination of oriented cycles of \emph{weight 1} no two of which are cyclically equivalent.

Two wQPs $(A,S),(A',S')$ are \emph{right-equivalent} if there is a $G$-graded algebra isomorphism $\varphi:R\left<\left<A\right>\right>\to R\left<\left<A'\right>\right>$ so that $\varphi(S)$ is cyclically equivalent to $S'$. 

\begin{defn}
A wQP $(A,S)$ is called \emph{trivial} if $S\in A^2_{cyc}$ and there exists a homogeneous basis $\{a_1,\cdots,a_m,b_1,\cdots,b_m\}$ of $A$ so that $S$ is cyclically equivalent to $\sum a_ib_i$.
\end{defn}

Since $Q$ has no loops, every arrow $a:i\to j$ goes \emph{forward} ($i<j$) or \emph{backwards} ($i>j$). Any homogeneous basis for $A$ consists of forward arrows $a_p$ and backward arrows $b_q$. Then any $S\in A^2_{cyc}$ is cyclically equivalent to a unique $S'$ of the form
\[
	S'=\sum c_{pq}a_pb_q
\]
We call $S'$ the \emph{forward-backward} form of $S$.
\begin{prop}
$(A,S)$ is trivial if and only if $(c_{pq})$ is an invertible square matrix.
\end{prop}

\begin{proof}
A homogeneous basis $B$ consists of bases $B_g^{ij}$ for $e_i A_ge_j$. Thus $(c_{pq})$ is a block diagonal matrix with one block for every $(i,j,g)$ with $i<j$. On one block, $(a_p)$ is a basis for $e_iA_ge_j$ and $(b_q)$ is a basis for $e_jA_{g^{-1}}e_i$. The potential is trivial if each block can be transformed into an identity matrix. This is equivalent to each block of $(c_{pq})$ being invertible which is equivalent to $(c_{pq})$ being invertible.
\end{proof}

\subsection{Splitting Theorem}

This theorem states that, for any wQP $(A,S)$, there is a decomposition
\[
	(A,S)\cong (A_{triv},S_{triv})\oplus (A_{red},S_{red})
\]
where $(A_{triv},S_{triv})$ is trivial and $(A_{red},S_{red})$ is \emph{reduced} which means $S_{red}\in\frak m^3$. The first step is to identify $A_{triv}$ and $S_{triv}$.

Given a wQP $(A,S)$, let $S_{triv}=S^{(2)}$ be the degree $2$ part of $S$. Let
\[
	S_g^{ij}\in e_iA_ge_j\otimes e_jA_{g^{-1}}e_i
\]
be the cyclically equivalent forward-backward $(i,j,g)$ component of $S^{(2)}$. Then $S_g^{ij}$ gives linear maps:
\[
	 (e_jA_{g^{-1}}e_i)^\ast\to e_iA_ge_j
\]
\[
	(e_iA_ge_j)^\ast\to  e_jA_{g^{-1}}e_i
\]
Let $V_g^{ij}\subseteq e_iA_ge_j$ and $V_{g^{-1}}^{ji}\subseteq e_jA_{g^{-1}}e_i$ be the images of these two linear maps. Then 
\[
	A_{triv}=\bigoplus_{i<j,g\in G} V_g^{ij}\oplus V_{g^{-1}}^{ji}
\]
This is independent of any choice of homogeneous basis for $A$. The following lemma, with proof left to the reader, implies that $(A_{triv},S_{triv})\subseteq (A,S)$ is invariant under $G$-graded automorphisms $\varphi$ of $R\left<\left<A\right>\right>$ in the sense that $(\varphi(A_{triv}),\varphi(S_{triv}))$ is the trivial part of $(A,\varphi(S))$.

\begin{lem}
Let $\varphi$ be a $G$-graded automorphism of $R\left<\left<A\right>\right>$ over $R$. Let $\varphi^{(1)}$ be the induced linear automorphism of $A=\frak m/\frak m^2$. Then 
\[
\varphi(S)^{(2)}=\left(\varphi^{(1)}\otimes \varphi^{(1)}\right)(S^{(2)}),
\]
i.e., the degree $2$ part of $\varphi(S)$ is the image of $S^{(2)}\in A_{cyc}^2\subset A^2$ under $\varphi^{(1)}\otimes \varphi^{(1)}:A^2\to A^2$.
\end{lem}

Let $A_{red}=A/A_{triv}$. This has a decomposition:
\[
	A_{red}=\bigoplus_{i,j,g} e_iA_ge_j/V_g^{ij}
\]

\begin{prop}
Choose a vector space complement $W_g^{ij}$ for $V_g^{ij}$ in $e_iA_ge_j$ for every $i,j,g$. Then there exists a $G$-graded unitriangular automorphism $\varphi$ of $R\left<\left<A\right>\right>$ so that $\varphi(A,S)$ is cyclically equivalent to $(A_{triv},S_{triv})\oplus (A_{red},S_{red})$.
\end{prop}

\begin{proof} Let $V=A_{triv}=\bigoplus V_g^{ij}$ and let $W=\bigoplus W_g^{ij}\cong A_{red}$. For all $g\in G$ and $i<j$, choose a basis $(a_p)$ for $V_g^{ij}$ and a dual basis $(b_p)$ for $V_{g^{-1}}^{ji}$ so that $S^{(2)}=\sum a_pb_p$. Then, up to cyclic equivalence we have:
\[
	S=\sum (a_pb_p+a_pu_p+v_pb_p)+S'
\]
where $S'\in R\left<\left<W\right>\right>_{cyc}$. Then $u_p,v_p$ have degree $\ge2$ and weight $g^{-1},g$ respectively.

Let $\varphi$ be the unitriangular $G$-graded automorphism of $R\left<\left<A\right>\right>$ given by $\varphi(a_p)=a_p-v_p$, $\varphi(b_p)=b_p-u_p$ and $\varphi|W=id_W$. Then
\[
	\varphi(S)=\sum a_pb_p+S'+\text{terms of degree} \ge4
\]
Repeating this procedure, terms of increasingly higher degrees are added to $\varphi(S)$. So, $\varphi(S)$ will converge to $S_{triv}=\sum a_pb_p\in R\left<\left<V\right>\right>_{cyc}$ plus a reduced element of $R\left<\left<W\right>\right>_{cyc}$.
\end{proof}

\section{Mutation of weighted quivers with potential}

We review the definition of mutation of quivers with potential and show that the mutation process respects the weights, i.e., every term in the new potential can be chosen to have trivial weight.

\subsection{Mutation of weighted quivers} The definition of mutation of a weighted quiver was given in the introduction. We review this with additional notation and an example.

Let $(Q,wt)$ be a weighted quiver having no loops or oriented 2-cycles. Let $k$ be a vertex of $Q$. We define $\tilde\mu_k(Q,wt)=(\tilde Q,\tilde{wt})$ to be the weighted quiver obtained from $Q$ by adding an arrow $[ab]$ from $i$ to $j$ for any pair of composable arrows $a:i\to k$, $b:k\to j$ with target $k$ and source $k$, then reversing the orientation of all arrows going to and from $k$. The reversed arrow are indicated by an asterisk: $a^\ast:k\to i$, $b^\ast:j\to k$. The weights of these new arrows are given by $wt([ab])=wt(a)wt(b)$ and $wt(x^\ast)=wt(x)^{-1}$.

The new arrows $[ab]:i\to j$ may form oriented 2-cycles with existing arrows $c:j\to i$. If these exists such a $c$ with $wt(c)=wt([ab])^{-1}$, the pair of arrows $[ab], c$ is eliminated. Repeating as often as possible we obtain the weighted quiver $\mu_k(Q,wt)$ which has no oriented 2-cycles of trivial weight. We call such a quiver \emph{weight reduced} and we call the procedure, going from $(\tilde Q,\tilde{wt})$ to $\mu_k(Q,wt)$, \emph{weight reduction}.

An example is given in Figure~\ref{Fig1}. The first quiver in Fig~\ref{Fig1} is $Q$, the quiver with vertices $1,2,3$ and four arrows $a:1\to 2$, $b:2\to 3$ and $c,d:3\to 1$. The second quiver is $\tilde Q$. This has a new arrow $[ab]:1\to 3$ and reverses the arrows $a,b$ producing $a^\ast:2\to 1$ and $b^\ast:3\to 2$.

\begin{figure}[htbp]
\begin{center}
\[
\xymatrix{
&
	2\ar[dr]^b &&&&  & 2\ar[ld]_{a^\ast} &&&&  & 2\ar[ld]_{a^\ast} & \\
1 \ar[ru]^a&& 3 \ar[ll]_c\ar@/^1pc/[ll]_d &\ \ar[r]^{\mu_2}&\ &1 \ar@/^1pc/[rr]^{[ab]}&& 3 \ar[ll]_c\ar@/^1pc/[ll]_d\ar[lu]_{b^\ast}&\ \ar[r]^{?}&\ &1 && 3\ar[ll]_d\ar[lu]_{b^\ast}
	}
\]
\caption{$wt([ab])=wt(a)wt(b)$, $wt(a^\ast)=wt(a)^{-1}$, $wt(b^\ast)=wt(b)^{-1}$. The 2-cycle $[ab]c$ in the middle can be cancelled if and only if $wt(c)=wt([ab])^{-1}$.}
\label{Fig1}
\end{center}
\end{figure}
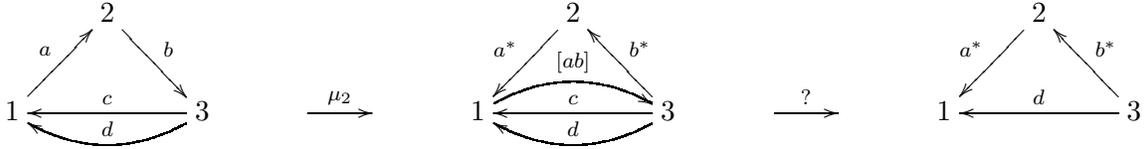

The {weight reduction step} is to eliminate any 2-cycles of trivial weight. If the weight of either $c$ or $d$ is inverse to $wt([ab])=wt(a)wt(b)$ then that arrow can be paired with $[ab]$ and eliminated. The final output is $\mu_k(Q,wt)$, the {weight reduction} of $(\tilde Q,\tilde{wt})$ which is shown in Fig~\ref{Fig1} assuming $wt([ab]c)=1$. 

\begin{defn}
A weighted quiver is call \emph{nondegenerate} if it has no loops or 2-cycles and, after any finite sequence of mutations, the mutated weighted quiver has no loops or 2-cycles.
\end{defn}

\subsection{Mutation of weighted quivers with potential} For a weighted quiver with potential, we will use the {mutation process dictated by the potential} which may not agree with the mutation process for weighted quivers without potential as outlined above. The difference is in the last step. To emphasize this we call the last step \emph{potential reduction}. This does not always agree with weight reduction.

For example, in Figure \ref{Fig1}, if we use the potential $S=abc$ then $[ab]c$ has trivial weight and would be eliminated by either weight reduction or potential reduction. However, if $S=0$, the 2-cycle $[ab]c$ cannot be eliminated by potential reduction even if it has trivial weight.

\begin{defn}
Let $(A,S)$ be a wQP with underlying weighted quiver $(Q,wt)$. Suppose $Q$ has loops or 2-cycles. Let $k\in Q_0$. Then the mutated weighted quiver with potential $(A',S')=\mu_k(A,S)$  with underlying weighted quiver $(Q',wt')$ (which is not necessarily equal to $\mu_k(Q,wt)$) is given in three steps. 
\begin{enumerate}
\item Construct the weighted quiver $\tilde\mu_k(Q,wt)=(\tilde Q,\tilde{wt})$ as described by the first two steps of the mutation process for weighted quivers.
\item Let $\tilde \mu_k(A,S)=(\tilde A,\tilde S)$ where:
\begin{enumerate}
	\item $\tilde A=\bigoplus_{g\in G}\tilde A_g$ where 
	\[
	\tilde A_g=(1-e_k)A_g(1-e_k)\oplus e_k(A_{g^{-1}})^\ast\oplus (A_{g^{-1}})^\ast e_k\oplus \bigoplus_{xy=g}A_xe_kA_y
	\]
	\item $\tilde S=[S]+\Delta_k$ where $[S]$ is $S$, expressed as an (infinite) linear combination of paths not starting or ending at $k$, with all pairs of consecutive arrows $a,b$ in the path passing through $k$ being replaced by $[ab]\in \tilde A$ and
	\[
		\Delta_k\in \bigoplus_{x,y\in G}  e_k(A_x)^\ast\otimes_R A_x e_k A_y\otimes_R (A_y)^\ast e_k=\bigoplus_{x,y\in G} \Hom_{R\otimes R^{op}}(A_xe_k\otimes e_kA_y,A_xe_k A_y)
	\]
	is the sum of terms corresponding to the canonical $R$-bimodule isomorphisms $A_xe_k\otimes e_kA_y\cong A_xe_k A_y$. (Each term is a sum of 3-cycles in $\tilde A$ of trivial weight.)
\end{enumerate}
\item Since every term in $\tilde S$ has trivial weight, the wQP $(\tilde A,\tilde S)$ decomposes as $(\tilde A_{triv},\tilde S_{triv})\oplus (\tilde A_{red},\tilde S_{red})$. Let:
\[
	\mu_k(A,S)=(\tilde A_{red},\tilde S_{red})
\]
with underlying weighted quiver $(Q',wt')$.
\end{enumerate}
\end{defn}

\begin{prop}\label{prop: Q' is partial weight reduction}
$\mu_k(Q,wt)$ is the weight reduction of $(Q',wt')$. So, $(Q',wt')=\mu_k(Q,wt)$ in the special case when $Q'$ has no oriented cycles.
\end{prop}

\begin{proof} By definition, $\mu_k(Q,wt)$ is obtained from $(\tilde Q,\tilde{wt})$ by removing as many oriented 2-cycles with trivial weight as possible. Each term in $\tilde S_{triv}$ is an oriented 2-cycle with trivial weight. So, $(Q',wt')$ is an intermediate step in the weight reduction process. Completion of the weight reduction process proves the proposition.
\end{proof}

\begin{defn}
A quiver with potential (without weights) is called \emph{nondegenerate} if, after any number of mutations, the resulting QP has no oriented 2-cycles.
\end{defn}

Now we quote one of the key theorems of \cite{DWZ}.

\begin{thm}\label{thm: existence of nondeg potential}
Any quiver $Q$ without loops or 2-cycles has a nondegenerate potential.
\end{thm}

\subsection{Sign coherence of $c$-vectors}

\begin{cor}
Let $(Q,wt)$ be a weighted quiver without loops or oriented 2-cycles so that every oriented cycle has trivial weight. Then $(Q,wt)$ is nondegenerate.
\end{cor}

\begin{proof}
Let $(A,S)$ be a nondegenerate potential for $Q$. Since every term in $S$ is an oriented cycle, $(A,S)$ is a weighted quiver with potential. Since $S$ is nondegenerate we have, by Proposition \ref{prop: Q' is partial weight reduction}, that mutation of the wQP $(A,S)$ is compatible with mutation of the weighted quiver $(Q,wt)$. So, no oriented 2-cycles of nontrivial weight will be produced in the mutation process.
\end{proof}

As outlined in the introduction, this proves sign coherence of $c$-vectors.

\begin{cor}
Let $Q$ be a quiver with frozen vertices all of which are sources (or all sinks). Then no sequence of mutations on nonfrozen vertices can produce an arrow between frozen vertices.
\end{cor}

\section{Application}

As another application of weighted quivers we classify all possible weights on tame quivers (quivers mutation equivalent to a tame acyclic quirver) up to equivalence (defined below). First we reduce to the case of $\widetilde A_{n-1}$.

\subsection{Equivalence of weights}

We will show that, outside of type $\widetilde A_{n-1}$, all weights on tame quivers are equivalent.

\begin{defn}
Two weight systems $wt,wt':Q_1\to G$ on the same quiver $Q$ are said to be \emph{(vertex) equivalent}, written $(Q,wt)\approx (Q,wt')$, if there exists a function $g:Q_0\to G$ so that, for any arrow $a:i\to j$ in $Q$, $wt'(a)=g(i)^{-1}wt(a)g(j)$.
\end{defn}

\begin{lem}\label{lem 01}
Any mutation of equivalent weight on a quiver $Q$ are equivalent, i.e., $\mu_v(Q,wt)\approx \mu_v(Q,wt')$ for any $v\in Q_0$ if $(Q,wt)\approx (Q,wt')$.
\end{lem}

\begin{proof}
The equivalence $\mu_v(Q,wt)\approx \mu_v(Q,wt')$ is given by the same vertex function $g:Q_0\to G$. Indeed, mutation produces inverted arrows $a^\ast:j\to i$ with 
\[
wt'(a^\ast)=[wt(a)']^{-1}=[g(i)^{-1}wt(a)g(j)]^{-1}=g(j)^{-1}wt(a^\ast)g(i)
\]
and composed arrows $[ab]:i\to k$ out of $a:i\to j$, $b:j\to k$ with
\[
	wt'([ab])=wt'(a)wt'(b)=[g(i)^{-1}wt(a)g(j)][g(j)^{-1}wt(b)g(k)]=g(i)^{-1}wt([ab])g(k)
\]
So, the mutated quivers are equivalent.
\end{proof}


\begin{lem}\label{lem 02}
If $Q$ is a tree, any weight on $Q$ is equivalent to the trivial weight, i.e., there is a function $g:Q_0\to G$ so that $wt(a)=g(i)^{-1}g(j)$ for all arrows $a:i\to j$. 
\end{lem}

\begin{proof}
Choose a vertex $v\in Q_0$. For any $i\in Q_0$, let $g(i)=wt(p)$ be the weight of the unique path $p$ from $v$ to $i$. If we extend this path by $a:i\to j$ we get $g(j)=wt(pa)=g(i)wt(a)$.
\end{proof}

\begin{thm}\label{thm: no nontrivial weights on Euclidean D,E}
Let $Q$ be a quiver mutation equivalent to a Euclidean quiver of type $\widetilde D$ or $\widetilde E$. Then any weight on $Q$ is equivalent to the trivial weight.
\end{thm}

\begin{proof}
This follows from Lemmas \ref{lem 01} and \ref{lem 02} since the quivers of $\widetilde D$ and $\widetilde E$ are trees.
\end{proof}

To find quivers with nontrivial weights, the following easy observation is very useful.

\begin{prop}\label{prop: weights on cycles}
Two equivalent weights on $Q$ take the same value on any cycle, oriented or not.\qed
\end{prop}

\subsection{Weights on quivers of type $\widetilde A_{n-1}$} Weighted quivers of type $\widetilde A_{n-1}$ are described by a list of local conditions and one global condition. We deal with the global condition first. The following definition is only appropriate when studying quivers of type $A_n$ and $\widetilde A_{n-1}$.


\begin{defn}\label{def: triangle, Euler char}
A \emph{triangle} in $(Q,wt)$ is defined to be an oriented 3-cycle in $Q$ with trivial weight.
The \emph{Euler characteristic} of any weighted quiver $(Q,wt)$ is defined to be
\[
	\chi(Q,wt)=|Q_0|-|Q_1|+|Q_2|
\]
where $Q_2$ is the set \emph{triangles} in $(Q,wt)$.
\end{defn}

Note that $\chi(Q,wt)$ is equal to the Euler characteristic of a topological space 
\[
	\chi(Q,wt)=\chi(B(Q,wt))
\]
where $B(Q,wt)$ is obtained from the underlying graph of $Q$ by attaching a 2-simplex to every triangle. (In other words, we ``fill in'' each triangle.)
For example, if $Q$ is a tree, its Euler characteristic is 1.

Euler characteristic can be used to check for certain combinatorial properties of the quiver, such as the number of \emph{triangle-free} ($\Delta$\emph{-free}) cycles which are defined to be simple cycles in $Q$ (those that do not go through the same vertex twice) that do not go through more than one edge of any triangle. $\Delta$-free cycles always come in pairs: $\gamma,\gamma^{-1}$.

\begin{rem}\label{rem: reduction process} Given a simple cycle $\gamma$ in $Q$ which is not $\Delta$-free and not a triangle, there is an obvious \emph{reduction process} given by replacing two consecutive edges in $\gamma$ which belong to the same triangle with the third edge of that triangle. The reduction process does not change the weight of the cycle since triangles have trivial weight by definition. Since reduction decreases the number of vertices by one, we eventually get either a triangle or a $\Delta$-free cycle. We call this final cycle $\overline \gamma$. In general, this might not be unique. 

We note that, in the topological space $B(Q,wt)$, $\gamma$ and $\overline\gamma$ are homotopic. Therefore, if $\gamma$ is not null-homotopic in $B(Q,wt)$, it reduces to a $\Delta$-free cycle $\overline \gamma$ with the same weight.
\end{rem}

\begin{prop}\label{prop: X=0 iff unique minimal cycle}
Let $(Q,wt)$ be a connected weighted quiver so that every arrow in $Q$ belongs to at most one triangle. Then the following are equivalent.
\begin{enumerate}
\item $\chi(Q,wt)=0$
\item There is a unique $\Delta$-free cycle in $Q$.
\end{enumerate}
\end{prop}

\begin{proof}
$(1)\then (2)$ We use the following construction. Let $Q'$ be obtained from $Q$ by deleting one edge from each triangle. If $Q$ has $k$ triangles then $Q'$ has no triangles and $k$ fewer edges. So $Q'$ is a graph with $\chi(Q')=\chi(Q,wt)=0$. Since $Q'$ is connected, it contains a unique simple cycle $\gamma$ (and its inverse). Choose $Q'$ so that $\gamma$ has minimal length. Then we claim that $\gamma$ is $\Delta$-free in $(Q,wt)$. Otherwise, $\gamma$ has two edges in one triangle and, changing the choice of $Q'$ on that triangle will give a shorter $\gamma$. So, $(Q,wt)$ has at least one $\Delta$-free cycle. To prove uniqueness, suppose that $(Q,wt)$ has two $\Delta$-free cycles $\gamma,\gamma'$ which are not inverse to each other. Then each triangle has at least one edge which is in neither $\gamma$ nor $\gamma'$. So, we $Q'$ can be chose so that it contains both $\gamma$ and $\gamma'$. But this contradicts the well-known elementary fact that a connected graph with $\chi=0$ has only one cycle.

$(2)\then(1)$ Any $\Delta$-free cycle in $(Q,wt)$ gives a cycle in some $Q'$ showing that $\chi(Q')=\chi(Q,wt)\le0$. If $\chi(Q')<0$ then $Q'$ has two cycles which are not homotopic in $B(Q,wt)$ (to each other or their inverses). The reduction process produces distinct $\Delta$-free cycles contrary to assumption. So, $\chi(Q,wt)=0$.
\end{proof}

\begin{rem}
We only need $(1)\then (2)$ from Proposition \ref{prop: X=0 iff unique minimal cycle}.
\end{rem}

\begin{lem}\label{lem: no trivial weight cycles}
Let $(Q,wt)$ be a connected weighted quiver with $\chi(Q,wt)=0$ so that no arrow belongs to more than one triangle. Suppose that the unique $\Delta$-free cycle in $Q$ has nontrivial weight. Then the only simple cycles with trivial weight are triangles.
\end{lem}

\begin{proof}
Let $\gamma$ be any simple cycle with trivial weight. Then the reduced cycle $\overline\gamma$ cannot be equal to the unique $\Delta$-free cycle since $\overline\gamma$ has trivial weight. So, $\overline\gamma$ must be a triangle. This implies $\gamma=\overline\gamma$ is a triangle since, otherwise, the last reduction step uses a triangle which meets $\overline\gamma$ in exactly one edge and that edge would belong to two triangles.
\end{proof}

The following is the weighted version of the well-known classification of cluster-tilted algebras of type $\widetilde A_{n-1}$. (See, e.g., \cite{Bastian}.)

\begin{defn}\label{def: Cn(t)} Let $t\in G$ be a fixed nontrivial element of a fixed nontrivial group $G$. Let $\cC_n(t)$ be the collection of connected weighted quivers $(Q,wt)$ with $n$ vertices satisfying the following:
\begin{enumerate}
\item No edge of $Q$ belongs to more than one triangle (Def. \ref{def: triangle, Euler char}).
\item $\chi(Q,wt)=0$.
\item The unique $\Delta$-free cycle of $(Q,wt)$ is unoriented and has weight $t$ (or $t^{-1}$). In particular, $Q$ has no loops and no oriented $2$-cycles.
\item Valency: every vertex of $Q$ has degree $\le 4$.
	\begin{enumerate}
	\item If vertex $v$ of $Q$ has degree 3, then $v$ belongs to one triangle.
	\item If vertex $v$ of $Q$ has degree 4, then $v$ belongs to two triangles.
	\end{enumerate}
\end{enumerate}
\end{defn}



\begin{prop}\label{prop: Cn(t) is mutation invariant}
$\cC_n(t)$ is invariant under mutation.
\end{prop}

\begin{proof} First we claim that mutation at any vertex $v$ does not change the Euler characteristic of $(Q,wt)$. If $v$ has degree 1, 3 or 4 this is clear since the number of vertices, edges and triangles remain the same. If $v$ is 2-valent, then the number of vertices remains the same but the number of edges and triangles either both increase by one or both decrease by one. So, $\chi(\mu_v(Q,wt))=\chi(Q,wt)=0$ in all cases. Condition (1) is also clearly preserved by mutation at any vertex since any new triangles contains that vertex and these new triangles satisfy (1) by construction.

To show (3), let $\gamma$ be the unique $\Delta$-free cycle of $(Q,wt)$. If the mutation vertex $v$ lies in a triangle and the arrow $\gamma$ opposite $v$ in the triangle lies in $\gamma$, then a new cycle $\gamma'$ for $\mu_v(Q,wt)$ is obtained by replacing $\alpha$ with the other two edges of the triangle with orientation reversed. Then $\gamma'$ will be $\Delta$-free and unoriented with the same weight at $\gamma$. If $v$ lies on $\gamma$, the two adjacent arrows in $\gamma$ have either the same or opposite orientation. In the first case, these two arrows in $\gamma$ will be replaced by a new arrow pointing in the same direction with the same weight giving again an unoriented $\Delta$-free cycle for $\mu_k(Q,wt)$ of the same weight as $\gamma$. In the second case, the direction of the two arrows are reversed and we obtain a new unoriented cycle $\gamma'$ for $\mu_v(Q,wt)$ with the same weight. If neither of these happens then $\gamma'=\gamma$ and there is nothing to prove. Therefore, $\mu_k(Q,wt)$ satisfies (3) in all cases.

Lastly, we show condition (4) is preserved. Suppose $Q$ is a weighted quiver in $\cC_n(t)$ and let $v$ be a vertex of valency 1. Then $\mu_v (Q)$ will be a quiver whose orientation and weight at the edge incident to $v$ has become inverted. So $\mu_v (Q) \in \cC_n(t)$. Now suppose $v$ has valency 2. Case 1 (edges do not belong to an oriented 3-cycle) If $v$ is a sink or source, then clearly the mutated quiver belongs to $\cC_n(t)$. However, if $v$ is neither a sink or source, then mutating $Q$ at $v$ produces an oriented 3-cycle of trivial weight. Case 2 (edges belong to a 3-cycle). Mutating at vertex $v$ kills the arrow opposite it and reverses the orientation of incident arrows. Thus, $\mu_v(Q) \in \cC_n(t)$. Now consider the case where $v$ is 3-valent. Mutating at $v$ will kill the arrow in the 3-cycle directly opposite to $v$, reverse the orientation of all three edges incident to $v$, and form another 3-cycle of trivial weight (by construction) with $v$ being one of the vertices of the cycle. So again we remain in $\cC_n(t)$. Finally, we treat the last case. Let $v$ be 4-valent. Then, mutation at $v$ will kill both edges opposite to $v$ and produce two oriented 3-cycles of trivial weight.
\end{proof}

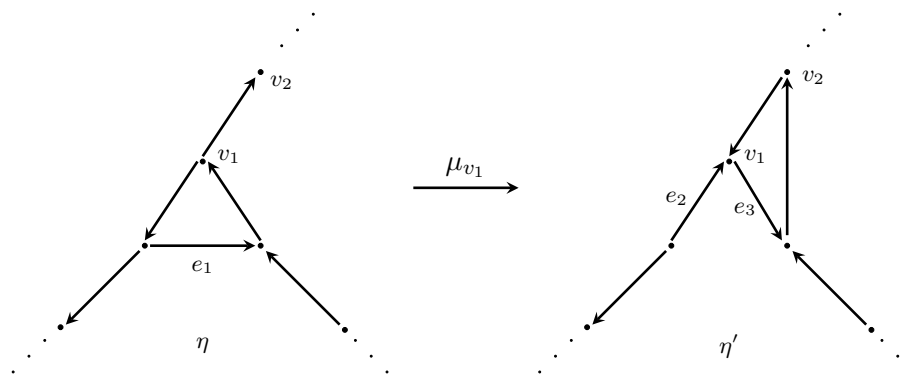
\begin{figure}
\begin{tikzpicture}[scale=0.7]




\begin{scope}

\draw [<-,>=stealth, line width= 1pt] (-7.6,-0.6) -- (-6.2,0.8) ;
\draw [<-,>=stealth, line width= 1pt] (-3.8,0.8) -- (-2.4,-0.6) ;

\draw [->,>=stealth, line width= 1pt] (-6,0.9) -- (-4,0.9) ;
\draw [->,>=stealth, line width= 1pt] (-3.9,1) -- (-4.9,2.5) ;
\draw [->,>=stealth, line width= 1pt] (-5.1,2.5) -- (-6.1,1) ;

\draw [->,>=stealth, line width= 1pt] (-5,2.6) -- (-4,4.1) ;

\draw[black, fill=black] (-6.1,0.9) circle[radius=1.3pt];
\draw[black, fill=black] (-3.9,0.9) circle[radius=1.3pt];
\draw[black, fill=black] (-5,2.5) circle[radius=1.3pt];
\draw[black, fill=black] (-3.9,4.2) circle[radius=1.3pt];
\draw[black, fill=black] (-7.7,-0.65) circle[radius=1.3pt];
\draw[black, fill=black] (-2.3,-0.7) circle[radius=1.3pt];

\begin{scope}[xshift=-1cm,yshift=-1.5cm]
\draw[black, fill=black] (-2.5,6.2) circle[radius=0.5pt];
\draw[black, fill=black] (-2.2,6.5) circle[radius=0.5pt];
\draw[black, fill=black] (-1.9,6.8) circle[radius=0.5pt];
\end{scope}

\draw[black, fill=black] (-8.6,-1.5) circle[radius=0.5pt];
\draw[black, fill=black] (-8.3,-1.2) circle[radius=0.5pt];
\draw[black, fill=black] (-8,-.9) circle[radius=0.5pt];

\draw[black, fill=black] (-2.1,-.9) circle[radius=0.5pt];
\draw[black, fill=black] (-1.8,-1.2) circle[radius=0.5pt];
\draw[black, fill=black] (-1.5,-1.5) circle[radius=0.5pt];

\draw  (-5,0.5) node {\footnotesize $e_1$};
\draw  (-4.5,2.6) node {\footnotesize $v_1$};
\draw  (-3.5,4) node {\footnotesize $v_2$};
\draw (-5,-1) node {\footnotesize $\eta$};

\begin{scope}[yshift=-1cm]
\draw [->,>=stealth, line width= 1pt] (-1,3) -- (1,3) node[midway, sloped,above, rotate = 0] {$\mu_{v_1}$};
\end{scope}

\end{scope}

\begin{scope}[xshift=10cm]

\draw [<-,>=stealth, line width= 1pt] (-7.6,-0.6) -- (-6.2,0.8) ;
\draw [<-,>=stealth, line width= 1pt] (-3.8,0.8) -- (-2.4,-0.6) ;


\draw [<-,>=stealth, line width= 1pt] (-5.1,2.5) -- (-6.1,1) ;

\draw [<-,>=stealth, line width= 1pt] (-5,2.6) -- (-4,4.1) ;
\draw [<-,>=stealth, line width= 1pt] (-3.9,4.1) -- (-3.9,1.1) ;
\draw [<-,>=stealth, line width= 1pt] (-4,1) -- (-4.9,2.5) ;



\draw[black, fill=black] (-6.1,0.9) circle[radius=1.3pt];
\draw[black, fill=black] (-3.9,0.9) circle[radius=1.3pt];
\draw[black, fill=black] (-5,2.5) circle[radius=1.3pt];
\draw[black, fill=black] (-3.9,4.2) circle[radius=1.3pt];
\draw[black, fill=black] (-7.7,-0.65) circle[radius=1.3pt];
\draw[black, fill=black] (-2.3,-0.7) circle[radius=1.3pt];

\begin{scope}[xshift=-1cm,yshift=-1.5cm]
\draw[black, fill=black] (-2.5,6.2) circle[radius=0.5pt];
\draw[black, fill=black] (-2.2,6.5) circle[radius=0.5pt];
\draw[black, fill=black] (-1.9,6.8) circle[radius=0.5pt];
\end{scope}

\draw[black, fill=black] (-8.6,-1.5) circle[radius=0.5pt];
\draw[black, fill=black] (-8.3,-1.2) circle[radius=0.5pt];
\draw[black, fill=black] (-8,-.9) circle[radius=0.5pt];

\draw[black, fill=black] (-2.1,-.9) circle[radius=0.5pt];
\draw[black, fill=black] (-1.8,-1.2) circle[radius=0.5pt];
\draw[black, fill=black] (-1.5,-1.5) circle[radius=0.5pt];

\draw  (-4.7,1.6) node {\footnotesize $e_3$};
\draw  (-6,1.8) node {\footnotesize $e_2$};
\draw  (-4.5,2.6) node {\footnotesize $v_1$};
\draw  (-3.4,4.1) node {\footnotesize $v_2$};
\draw (-5,-1) node {\footnotesize $\eta'$};
\end{scope}

\end{tikzpicture}
\caption{Proof of Proposition \ref{prop: Cn(t) equivalent to one cycle}: mutation at $v_1$ increases the length of the minimum cycle $\eta$. $wt(\eta)=wt(\eta')$ since $wt(e_1)=wt(e_2)wt(e_3)$.} \label{fig:M2}
\end{figure}

\begin{prop}\label{prop: Cn(t) equivalent to one cycle}
Any weighted quiver $Q \in \cC_n(t)$ is mutation equivalent to an unoriented $n$-cycle of weight $t$ or $t^{-1}$.
\end{prop}

\begin{proof}
We prove this by induction on $n$ minus the length of the minimal cycle in $Q$. If the minimum cycle has length $n$ then we are done. Otherwise, there must be a triangle with one side on the minimum cycle. Mutation at the opposite vertex will increase the length of the minimal cycle by 1 and the weight of the cycle will be unchanged. See Figure \ref{fig:M2}. By induction $(Q,wt)$ is mutation equivalent to a single unoriented cycle.
\end{proof}

As a consequence of Prop \ref{prop: Cn(t) is mutation invariant} and Prop \ref{prop: Cn(t) equivalent to one cycle}, we have the following theorem:

\begin{thm}\label{thm: Cn(t) is class}
The union of all $\cC_n(t)$ for all nontrivial $t$ is the class of weighted quivers of type $\widetilde{A}_n$ with nontrivial nondegenerate weights.
\end{thm}

\begin{rem}By Theorem \ref{thm: no nontrivial weights on Euclidean D,E}, these are all the tame weighted quivers with nontrivial nondegenerate weight.
\end{rem}



\begin{thebibliography}{aa}

\bibitem{Bastian} J. Bastian, \emph{Mutation classes of $\widetilde A_n$-quivers and derived equivalence classification of cluster-tilted algebras of type $\widetilde A_n$}, Algebra Number Theory, 5(5):567--594, 2011.

\bibitem{DWZ} H. Derksen, J. Weyman, and A. Zelevinsky, Quivers with potentials and their representations I: Mutations. \emph{arXiv: 0704.0649v4 }, 2008.


\end{thebibliography}
\end{document}